\documentclass[11pt]{article}

\usepackage{times}
\usepackage{amsthm}
\usepackage{amsmath}
\usepackage{amssymb}
\usepackage{mathrsfs}

 \def\X{\mathcal X} \def\C{\mathbb{C}} 
\def\N{\mathbb{N}}   
   
  \def\H{\mathcal H} 
 
\def\S{\mathcal S} \def\SS{\mathcal{S}}\def\G{\mathcal{G}}

    \def\e{\mathfrak e}

\def\I{{\rm 1\kern-.26em I}}

\def\Op{\mathfrak{Op}}

\def\X{\mathcal X}

\def\1{\mathfrak{1}}
\def\LL{\mathfrak L}

\def\M{\mathbf M}
\def\N{\mathbf N}
\def\R{\mathbf R}

\def\I{I\hspace{-1.7pt}I\hspace{-1pt}{\rm d}}

\def\({\left(}
\def\){\right)}
\def\[{\left[}
\def\]{\right]}
\def\<{\left<}
\def\>{\right>}

\newtheorem{Theorem}{Theorem}[section]
\newtheorem{Remark}[Theorem]{Remark}
\newtheorem{Lemma}[Theorem]{Lemma}

\newtheorem{Corollary}[Theorem]{Corollary}
\newtheorem{Proposition}[Theorem]{Proposition}
\newtheorem{Definition}[Theorem]{Definition}
\newtheorem{Example}[Theorem]{Example}

\numberwithin{equation}{section}

\begin{document}

\title{Abstract composition laws\\ and their modulation spaces}

\author{Marius M\u antoiu $^1$   and Radu Purice$^2$}

\maketitle

\footnote*{


\begin{itemize}
\item[$^1$] Departamento de Matem\'aticas, Universidad de Chile, Las Palmeras 3425, Casilla 653,
Santiago, Chile, \\
Email: {\tt mantoiu@uchile.cl}
\item[$^2$] Institute
of Mathematics Simion Stoilow of the Romanian Academy, P.O.  Box
1-764, Bucharest, \\ RO-70700, Romania,
Email: {\tt radu.purice@imar.ro}

\end{itemize}}


\maketitle

\begin{abstract}
On classes of functions defined on $\mathbb R^{2n}$ we introduce
abstract composition laws modelled after the pseudodifferential
product of symbols. We attach to these composition laws modulation
mappings and spaces with useful algebraic and topological
properties.
\end{abstract}

\section*{Introduction}\label{intro}

Very roughly, the basic part of a pseudodifferential calculus consists
in a prescription $\mathfrak{Op}$ to transform suitable functions (symbols) defined on the "phase space" $\Xi=\mathbb R^{2n}$
into operators acting on functions defined on "the configuration space" $\X=\mathbb R^n$.
This procedure produces a non-commutative composition law $\sharp$ on symbols, the Moyal product,
emulating the operator product. The prototype that we have in view is the Kohn-Nirenberg pseudodifferential calculus \cite{KN}
and its symmetric form, the Weyl calculus \cite{Fo}.
Some more general pseudodifferential  calculi have been proposed in the literature. In \cite{MP1} we introduced the magnetic Weyl calculus,
which is a generalization of the usual pseudodifferential theory in Weyl form to the case when a magnetic field
(a closed $2$-form on $\X$) is present (cf. \cite{KO1,MPR2,IMP}).
See also \cite{BB1,BB2,BB5} for an extension to nilpotent Lie groups, building on \cite{Pe}.

Modulation spaces are Banach spaces of functions introduced by H. Feichtinger \cite{F1,F2}.
They evolved especially in connection with Time Frequency Analysis, Gabor Frames and Signal Processing
Theory.  Lately their importance in the theory of
pseudodifferential operators has been discovered  and the interconnection between modulation spaces and
pseudodifferential theory developed considerably. We
cite, without any claim of completeness, \cite{B1,FGL,G1,G2,G3,G4,HRT,S1,S2,T1,T2}; see also references therein.
They can be considered an attractive alternative to the standard theory \cite{Fo}  of H\"ormander-type symbol spaces. In many circumstances
they lead to sharper results, simpler proofs or finer insight.

The point of view that we want to emphasize in the present paper is that behind all these developments one can find an interesting
algebraic structure based on the Moyal product, leading to a version of the modulation spaces.
It comes out in fact that explicit expressions for the composition law $\sharp$ and the representation
$\mathfrak{Op}$ are not essential; some general properties are enough to develop at least the basic aspects of the theory.
So we are going to work with an abstract associative
product $\#$ on $\mathcal{S}(\Xi)$ satisfying some suitable but not very restricted properties. To this product we assign several
canonical transformations (one of them is called the modulation mapping) having good algebraic and topological properties.
A localization procedure (choice of a window) coupled with these transformations leads to a mechanism of inducing spaces
of distributions on $\Xi$ from convenient spaces of distributions on $\Xi\times\Xi$.
In a later paper, representations of the emerging algebras will be studied, as an abstractization of the representation $\mathfrak{Op}$.

A short description of our paper follows.

In the first section we introduce convenient composition
laws $\#$ on the Schwartz space $\mathcal{S}(\Xi)$, generalizing the usual Weyl product $\sharp$. Following \cite{GBV,MP1},
we extend $\#$ to certain classes of distributions by using duality techniques.
We also mention briefly representations in Hilbert and locally convex spaces, mainly to justify one of the axioms we impose on $\#$.
Part of any modulation theory involves switching to functions defined on
"the doubled space" $\Xi\times\Xi$. By a tensor construction, we extend the composition to this setting. We also introduce
a "crossed-product" multiplication, playing an important role in the next chapters.

We come then to the usual strategy of defining modulation mappings and spaces by doubling the number of variables and reducing back by
means of a window.
Topological isomorphisms $\M,\N,\R$ are defined on the Schwartz space $\mathcal{S}(\Xi\times\Xi)$
and then extended by duality to the space of tempered distributions $\mathcal{S}^\prime(\Xi\times\Xi)$.
They are determined by the composition law $\#$ only and they are connected to each other by some simple transformations.
We are going to work with $\M$ mainly (it will be called {\it the} modulation mapping).
Under a new requirement on $\#$, it satisfies a useful
orthogonality condition. It also defines a $^*$-algebra isomorphism from structures built on doubling the initial
composition law $\#$ to structures involving the simple crossed-product multiplication.

Applying $\M$ to elementary tensors $F=f\otimes h$ and freezing the window $h$ (often supposed to be real and idempotent),
we get an efficient tool to
induce spaces, norms and properties from function spaces $\LL$ defined on $\Xi\times\Xi$. We study the basic features of this induction
procedure. Let us stress that some properties of $\LL$ (as being a $^*$-algebra, for instance) become universal:
the induced spaces $\LL^\M_h\equiv\LL^\#_h$ will have the same property for all the good composition laws $\#$ and windows $h$.
To confer to this idea the full technical strengths we shall discuss briefly Moyal algebras,
having \cite{GBV} as a source of inspiration (see also \cite{MP1}).
Invariance under the change of window or of the modulation mapping are also addressed.
We touch very superficially the problem of spectral invariance problem.

We give some examples in the last section. The first one is merely a counter-example: the point-wise multiplication does not fit
in our scheme. We show then that the standard Weyl calculus satisfies the axioms and that the emerging modulation spaces
associated to $\LL=L^{p,q}(\Xi\times\Xi)$
coincide with the traditional ones, based on the Short Time Fourier Transform. The magnetic form of pseudodifferential theory constitutes
a final example, showing that results from \cite{MP4} can be covered by our unified treatment.

We hope to dedicate a future publication to a deeper study of the framework, including representations, spectral invariance,
concrete modulation spaces or more general examples.
It is very plausible that most of our developments also hold for other topological vector spaces than $\S'$,
on general abelian locally compact group $\Xi$ (see \cite{GS}) or even on less structured spaces.

\section{Abstract composition laws}\label{mapsymb}

\subsection{Framework, conventions and technical facts}\label{mapsymbobo}

The starting point is the euclidean space $\X:=\mathbb R^n$. We denote by $\X'$ the dual space of $\X$; the duality is given simply by
$\X\times\X'\ni(x,\xi)\mapsto x\cdot\xi$ (the canonical euclidean scalar product on $\mathbb{R}^n$).
The phase space will be $\Xi:=T^*\X\equiv\X\times\X'$, containing points
$X=(x,\xi)$, $Y=(y,\eta)$, $Z=(z,\zeta)$. It is a symplectic vector space with symplectic form
\begin{equation}\label{simp}
\sigma(X,Y)\equiv\sigma[(x,\xi),(y,\eta)]:=y\cdot \xi-x\cdot\eta.
\end{equation}

On tempered distributions on $\Xi$ we use the symplectic Fourier transformation
$$
(\mathfrak F f)(Y)\equiv\widehat f(Y):=\int_\Xi dZ\,e^{-i\sigma(Y,Z)}f(Z).
$$
With a good choice of the Haar measure $dX$ on $\Xi$, it is unitary in $L^2(\Xi)$ and satisfies $\mathfrak F^2=1$.

We denote by $\mathbb{B}\,(\mathcal{U};\mathcal{V})$ the space of linear continuous operators between the topological vector spaces
$\mathcal{U}$ and $\mathcal{V}$. If $\mathcal{U}=\mathcal{V}$, we set simply
$\mathbb{B}\,(\mathcal{U};\mathcal{U})=:\mathbb{B}\,(\mathcal{U})$.
On such operator spaces we consider usually the topology of uniform convergence on bounded sets.

The Schwartz space $\mathcal{S}(\Xi)$ being a Fr\'{e}chet space, it is barrelled  \cite[p. III.25]{Bou}.
As $\mathcal{S}(\Xi)$ is also reflexive, \cite[p. IV.23]{Bou} implies that its strong dual
$\mathcal{S}'(\Xi)$ is bornologic and barrelled. They are both nuclear, Montel space. It seems rather plausible that such abstract
properties would be enough to develop the theory below in an even more general setting.

Given a Hausdorff Fr\'{e}chet space and its strong dual, the polars of the neighborhoods of 0 in any of the two spaces is
a basis for the bornology of the other and the polars of the bounded sets of any of the spaces is a basis for the neighborhoods
of $0$ in the other space. Recall that the Closed Graph Theorem holds in Fr\'echet spaces.

We are going to use the "real" scalar product $\langle f,g\rangle:=\int_\Xi dX f(X)g(X)$ and its natural extension to a duality
form  (on $\mathcal{S}'(\Xi)\times\mathcal{S}(\Xi)$ or $\mathcal{S}(\Xi)\times\mathcal{S}'(\Xi)$ for example).

Let us recall the canonical isomorphisms
\begin{equation}\label{tens1}
\mathcal{S}(\Xi\times\Xi)\cong \mathcal{S}(\Xi)\widehat\otimes\,\mathcal{S}(\Xi),\ \ \ \mathcal{S}'(\Xi\times\Xi)\cong
\mathcal{S}'(\Xi)\widehat\otimes\,\mathcal{S}'(\Xi).
\end{equation}
The symbol $\widehat\otimes$
stands for either the injective or the projective tensor product (by nuclearity).
In particular, the algebraic tensor product $\mathcal{S}(\Xi)\odot\mathcal{S}(\Xi)$ is dense in $\mathcal{S}(\Xi\times\Xi)$ and
$\mathcal{S}^\prime(\Xi)\odot\mathcal{S}'(\Xi)$ is dense in $\mathcal{S}'(\Xi\times\Xi)$. We also mention the Hilbert space isomorphism
\begin{equation}\label{tens2}
L^2(\Xi\times\Xi)\cong L^2(\Xi)\otimes L^2(\Xi).
\end{equation}

\subsection{Symbol composition on phase space}\label{pozution}

On the space $\mathcal{S}(\Xi)$ we shall  consider a bilinear associative composition law denoted by $\#$.
The following assumption will always stand:

\medskip
\noindent
{\bf Hypothesis A.}
\begin{enumerate}
\item
 $\mathcal{S}(\Xi)$ is a $^*$-algebra with the separately continuous composition law
$\#:\mathcal{S}(\Xi)\times \mathcal{S}(\Xi)\rightarrow \mathcal{S}(\Xi)$
and with the involution defined by complex conjugation $f(\cdot)\mapsto f^\#(\cdot):=\overline{f(\cdot)}$.
\item For any $f,g\in \mathcal{S}(\Xi)$ one has
\begin{equation}\label{crichi}
\int_\Xi dX (f\#g)(X)=\int_\Xi dXf(X)g(X).
\end{equation}
\end{enumerate}

\noindent
Corollary of Theorem 34.1 in \cite{Tr} implies the joint continuity of the map $\#$.
Note that the zero composition law $(f,g)\mapsto f\#g:=0$ does not verify (\ref{crichi}), so it is outside the scope of this paper.

Most of the time we will use (\ref{crichi}) in the form given by the next result.

\begin{Proposition}
 Under Hypothesis A, for $f_1,f_2,f_3\in \mathcal{S}(\Xi)$ one has {\rm the cyclicity condition}
\begin{equation}\label{croco}
\<f_1\# f_2,f_3\>=\<f_1,f_2\# f_3\>=\<f_2,f_3\# f_1\>.
\end{equation}
\end{Proposition}

\begin{proof}
Using (\ref{crichi}) and the associativity of $\#$ one has
$$
\begin{aligned}
\<f_1\# f_2,f_3\>&=\int_\Xi dX\big(f_1\# f_2\big)(X)f_3(X)=\int_\Xi dX\big[\big(f_1\# f_2\big)\# f_3\big](X)\\
&=\int_\Xi dX\big[f_1\# \big(f_2\# f_3\big)\big](X)=\int_\Xi dXf_1(X)\big(f_2\# f_3\big)(X)\\
&=\<f_1,f_2\# f_3\>.
\end{aligned}
$$
The second relation follows similarly, using also the commutativity of the ordinary product.
\end{proof}

In such a framework, we can apply a simple strategy \cite{GBV,MP1} to extend the product $\#$ to much larger spaces.
Using Hypothesis A, the linear maps
$$
L:\ \mathcal{S}(\Xi)\ni f\mapsto L_f\in\mathbb{B}\,[\mathcal{S}(\Xi)],\quad L_f(g):=f\#g,
$$
$$
R:\ \mathcal{S}(\Xi)\ni f\mapsto R_f\in\mathbb{B}\,[\mathcal{S}(\Xi)],\quad R_f(g):=g\#f.
$$
are well-defined.

\begin{Proposition}
The linear maps $L,R:\mathcal{S}(\Xi)\rightarrow\mathbb{B}\,[\mathcal{S}(\Xi)]$ are continuous when on $\mathbb{B}\,[\mathcal{S}(\Xi)]$
we consider the topology of uniform convergence on bounded subsets of $\mathcal{S}(\Xi)$.
\end{Proposition}

\begin{proof}
Let us consider a neighborhood of $0\in\mathbb{B}\,[\mathcal{S}(\Xi)]$ of the form
$$
V_{B,U}:=\left\{\,T\in\mathbb{B}\,[\mathcal{S}(\Xi)]\ \mid\ Tf\in U,\,\forall f\in B\right\},
$$
for some $B\subset\mathcal{S}(\Xi)$ bounded and $U$ a neighborhood of $0\in\mathcal{S}(\Xi)$.
Due to the joint continuity of the composition $\#$, we deduce the existence of two neighborhoods $U_1,U_2$ of $0\in\mathcal{S}(\Xi)$
such that $U_1\#U_2\subset U$, i.e.
$$
\forall g_1\in U_1,\ \forall g_2\in U_2,\ \ L_{g_1}(g_2)=R_{g_2}(g_1)\in U.
$$
But $B$ being bounded, there exist two numbers $\lambda_1>0$ and $\lambda_2>0$ such that
$\lambda_1B\subset U_1$ and $\lambda_2B\subset U_2$. Thus
$$
\begin{array}{rcl}
f_1\in U_1,\ f_2\in B&\Rightarrow& \big(\lambda_2L_{f_1}\big)(f_2)=\lambda_2\big(f_1\#f_2\big)=
f_1\#\big(\lambda_2f_2\big)\in U,\\
f_1\in B,\ f_2\in U_2&\Rightarrow&\big(\lambda_1R_{f_2}\big)(f_1)=\lambda_1\big(f_1\#f_2\big)=
\big(\lambda_1f_1\big)\#f_2\in U.
\end{array}
$$
Therefore $\lambda_2U_1\subset L^{-1}\big(V_{B,U}\big)$ and $\lambda_1U_2\subset R^{-1}\big(V_{B,U}\big)$, just noticing that
$\lambda L_f=L_{\lambda f}$ and $\lambda R_f=R_{\lambda f}$. This shows that $L$ and $R$ are continuous.
\end{proof}

For each $f\in\mathcal{S}(\Xi)$ the maps $L_f,R_f:\S(\Xi)\rightarrow\S(\Xi)$ have associated transposed maps $L_f^{\,t}$ and $R_f^{\ t}$
that are in $\mathbb{B}\,[\mathcal{S}'(\Xi)]$ for the weak topology on $\mathcal{S}'(\Xi)$,
but also for the strong topology on the dual (Corollary of Proposition 19.5 in \cite{Tr}).
By the cyclicity property \eqref{croco}, $L^{\,t}_f$ can be regarded as an extension
of $R_f$ and $R^{\ t}_f$ as an extension of $L_f$. So, for $f\in\S(\Xi)$, let us define the linear maps
$$
R^\dagger:\mathcal{S}(\Xi)\rightarrow\mathbb{B}\,[\mathcal{S}'(\Xi)]:\quad R^\dagger(f):=L^{\,t}_f\,,
$$
$$
L^\dagger:\mathcal{S}(\Xi)\rightarrow\mathbb{B}\,[\mathcal{S}'(\Xi)]:\quad L^\dagger(f):=R^{\ t}_f\,.
$$

\begin{Proposition}\label{cont-dagger}
The linear applications $L^\dagger$ and $R^\dagger$ are continuous if on $\S'(\Xi)$ one considers the strong dual topology.
\end{Proposition}

\begin{proof}
We shall only treat $R^\dagger$. Consider a neighborhood of $0\in\mathbb{B}\,[\mathcal{S}'(\Xi)]$ of the form
$$
V_{\tilde{B},\tilde{U}}:=\left\{T\in\mathbb{B}\,[\mathcal{S}'(\Xi)]\ \mid\ Tf\in\tilde{U},\,\forall f\in\tilde{B}\right\}
$$
with some $\tilde{B}\subset\mathcal{S}'(\Xi)$ bounded and $\tilde{U}\subset\mathcal{S}'(\Xi)$ a neighborhood of $0\in\mathcal{S}'(\Xi)$.
We can take $\tilde{U}$ of the form
$$
\tilde{U}_{B,\epsilon}:=\left\{f\in\mathcal S'(\Xi)\ \mid\ |\<f,h\>|<\epsilon,\ \forall\,h\in B\right\},
$$
for $B\subset\mathcal S(\Xi)$ bounded and $\epsilon>0$.

One has $\left<R^\dagger(g)f,h\right>=\left<L^{\,t}_g f,h\right>=\left<f,g\# h\right>$, so the relation
$R^\dagger(g)\in V_{\tilde{B},\tilde{U}_{B,\epsilon}}$ means
$$
\left| \left<f,g\# h\right>\right|<\epsilon,\ \ \forall\,h\in B\subset\mathcal S(\Xi),\ \forall f\in\tilde{B}\subset\mathcal{S}'(\Xi).
$$
Since $\mathcal S(\Xi)$ is a Hausdorff Fr\'{e}chet space, the bounded set
$\tilde{B}\subset\mathcal{S}'(\Xi)$ is contained in the polar $W^\circ$ of a neighborhood $W$ of $0\in\mathcal S(\Xi)$, i.e.:
$$
f\in\tilde{B}\ \Rightarrow\left| \left<f,k\right>\right|<1,\ \forall\,k\in W.
$$
Due to the joint continuity of the composition $\#$, there exist two neighborhoods $U_1$ and $U_2$ of $0\in\mathcal S(\Xi)$ such that
$U_1\#\,U_2\subset W$. Thus
$$
f\in\tilde{B},\,g_1\in U_1,\,h\in U_2\ \Rightarrow\ \left| \left<f,\epsilon\big(g_1\# h\big)\right>\right|<\epsilon.
$$
But since $B\subset\mathcal S(\Xi)$ is bounded, there exists a positive constant $\mu>0$ such that $\mu B\subset U_2$.
Noticing that $\epsilon\big[g_1\#(\mu h)\big]=(\epsilon \mu g_1)\# h$, we conclude that for any $g\in\epsilon \mu U_1$ we have
$R^\dagger(g)\in V_{\tilde{B},\tilde{U}_{B,\epsilon}}$, so $R^\dagger$ is continuous.
\end{proof}

To conclude the discussion, it is natural now to define the bilinear composition laws
$$
\mathcal S(\Xi) \times\mathcal S'(\Xi)\ni(g,f)\mapsto g\#'\,f:=L^\dagger(g)f\in\mathcal S'(\Xi),
$$
$$
\mathcal S'(\Xi) \times\mathcal S(\Xi)\ni(f,g)\mapsto f\ '\# g:=R^\dagger(g)f\in\mathcal S'(\Xi).
$$
The cyclicity property (\ref{croco}) implies that both $\#'$ and $'\#$ are extensions of the application
$\#:\mathcal S(\Xi)\times\mathcal S(\Xi)\rightarrow\mathcal S(\Xi)$. Therefore we simplify notations writing $\#$
instead of $\#'$ or $'\#$. In fact, for $f\in\S'(\Xi),\,g,h\in\S(\Xi)$, one has
\begin{equation}\label{maiclar}
\<f\#g,h\>:=\<f,g\#h\>,\ \quad\<g\#f,h\>:=\<f,h\#g\>,
\end{equation}
where $\<\cdot,\cdot\>$ is now interpreted as the duality between $\S'(\Xi)$ and $\S(\Xi)$.

By Proposition \ref{cont-dagger} and by some trivial manipulations of the definitions one gets

\begin{Corollary}\label{ciuhap}
The two mappings $\mathcal S(\Xi) \times\mathcal S'(\Xi)\overset{\#}{\mapsto} \mathcal S'(\Xi)$
and $\mathcal S'(\Xi) \times\mathcal S(\Xi)\overset{\#}{\mapsto} \mathcal S'(\Xi)$ are bilinear and separately continuous when
on $\S'(\Xi)$ we consider the strong topology. The same is true for the weak topology.
\end{Corollary}

It is seen immediately from (\ref{croco}) by approximation, that for $f_1,f_2\in\S(\Xi)$ and $g\in\S'(\Xi)$ one has the associativity relation
$(f_1\# g)\#f_2=f_1\#(g\# f_2)$ (identity in $\S'(\Xi)$). Denoting by $I\in\S'(\Xi)$ the constant function $I(X):=1,\ \forall X\in\Xi$, we get
easily from (\ref{crichi}) that $I\#f=f=f\#I,\ \forall f\in\S(\Xi)$.

\medskip
We recall the family $(\e_X)_{X\in\Xi}$ of functions that will play an important role in the sequel
\begin{equation}
\e_X(Z):=\exp\{i\sigma(X,Z)\},\ \ Z\in\Xi.
\end{equation}
In fact we have $e_X\ =\ \mathfrak F\delta_X$, with $\delta_X$ the point Dirac measure of mass $1$ concentrated in $X\in\Xi$.
Notice the relations
$$
\<\e_{-X}\#g,h\>=[\mathfrak F(g\# h)](X)\ \ {\rm and}\ \ \<g\#\,\e_{-X},h\>=[\mathfrak F(h\# g)](X)
$$
valid for $X\in\Xi$ and $g,h\in\mathcal{S}(\Xi)$. The next result is basically Plancherel's formula.

\begin{Lemma}\label{basica}
One has in weak sense
\begin{equation}\label{pleasca}
\int_\Xi dZ\,|\e_{-Z}\rangle\langle\e_{Z}|=1,\ \quad\int_\Xi dZ\,|\delta_Z\rangle\langle\delta_{Z}|=1.
\end{equation}
\end{Lemma}

\begin{proof}
We compute for $f,g\in L^2(\Xi)$
$$
\begin{aligned}
&\int_\Xi dZ\,\langle f,\e_{-Z}\rangle\langle\e_{Z},g\rangle=\int_\Xi dZ\,(\mathfrak F f)(Z)\,(\mathfrak F g)(-Z)\\
&=\int_\Xi dZ\,(\mathfrak F f)(Z)\,\overline{(\mathfrak F \overline g)(Z)}=\int_\Xi dZ\,f(Z)\,g(Z)=\langle f,g\rangle.
\end{aligned}
$$
\end{proof}

\subsection{Representations and a new assumption}\label{pozation}

We come now to a new assumption (cf. also \cite{T3,T4}) that will be needed in most situations.

\medskip
\noindent
{\bf Hypothesis B.} One has $\mathcal{S}(\Xi)\#\,\mathcal{S}'(\Xi)
\#\mathcal{S}(\Xi)\subset\mathcal{S}(\Xi)$.

\medskip
This assumption is not fulfilled for the point-wise product $(f_1\#f_2)(X):=f_1(X)f_2(X)$. However, it is verified for
many "Weyl-type" products.

\begin{Remark}\label{totusi}
Under Hypothesis B, for every $f_1,f_2\in\S(\Xi)$ and $g_1,g_2\in\S'(\Xi)$ we get by approximation
\begin{equation}\label{chavez}
\<g_1,f_1\# g_2\# f_2\>=\<f_2\# g_1\#f_1,g_2\>.
\end{equation}
Sometimes we shall denote this common value by $\<g_1\#f_1,g_2\# f_2\>$ of by $\<f_2\# g_1,f_1\# g_2\>$.
Using (\ref{chavez}), it follows easily that the trilinear mapping
$$
\S(\Xi)\times\S'(\Xi)\times\S(\Xi)\ni(f_1,g,f_2)\mapsto f_1\# g\# f_2\in\S'(\Xi)
$$
is separately continuous for either of the two interesting topologies on $\S'(\Xi)$.
\end{Remark}

Under Hypothesis B one can also define
\begin{equation}\label{dreptul}
\S'(\Xi)\times\S(\Xi)\times\S'(\Xi)\ni(g_1,f,g_2)\mapsto g_1\# f\# g_2\in\S'(\Xi)
\end{equation}
simply by setting
\begin{equation}\label{strambul}
\<g_1\# f\#g_2,h\>:=\<g_1,f\# g_2\# h\>,\ \quad\forall\,h\in\S(\Xi).
\end{equation}
Continuity properties are not difficult to prove.

To see why we expect Hypothesis B to hold, we mention very briefly representations. They will be treated systematically elsewhere,
but one might have them in mind for insight; actually very often they predate and motivate the algebraic structure defined
by a composition law $\#$.
We assume that we are given a triple of spaces $\mathcal G\hookrightarrow\H\hookrightarrow\G'$
and a representation $\mathfrak{Op}$ of the $^*$-algebra $(\mathcal{S}(\Xi),\#,^\#)$ compatible with this triple.
In detail, this means the following:

\medskip
\noindent
{\bf Hypothesis X.}
\begin{enumerate}
\item
$\H$ is a complex, separable Hilbert space with scalar product $(\cdot,\cdot)$.
\item
$\G$ is a Fr\'echet space, continuously and densely embedded in $\H$.
\item
$\G'$ is the topological dual of $\G$, with the usual strong topology (the topology of uniform convergence on bounded subsets
of $\G$); then it contains naturally $\H$ continuously and densely.
\item
$\mathfrak{Op}:\mathcal{S}(\Xi)\rightarrow\mathbb B(\G',\G)$ is a topological isomorphism and a
$^*$-morphism:
$$
\mathfrak{Op}(f\# g)=\mathfrak{Op}(f)\mathfrak{Op}(g),\ \ \ \ \ \mathfrak{Op}(\overline f)=\mathfrak{Op}(f)^*,
\ \ \ \ \ \forall f,g\in\mathcal{S}(\Xi).
$$
\item
$\mathfrak{Op}$ extends to a topological isomorphism $:\mathcal{S}'(\Xi)\rightarrow\mathbb B(\G,\G')$.
\end{enumerate}

Then Hypothesis X implies Hypothesis B, since $\mathbb B(\G,\G')\mathbb B(\G',\G)\mathbb B(\G,\G')\subset\mathbb B(\G,\G')$.

\subsection{Composition laws on the double phase-space}\label{pozition}

Let us now raise $\#$ to the tensor product $\mathcal{S}(\Xi)\widehat\otimes\,\mathcal{S}(\Xi)\cong\mathcal{S}(\Xi\times\Xi)$
and introduce the composition law
$\square:\mathcal{S}(\Xi\times\Xi)\times \mathcal{S}(\Xi\times\Xi)\rightarrow
\mathcal{S}(\Xi\times\Xi)$,
uniquely defined by
\begin{equation}\label{dublu}
(f\otimes h)\square(g\otimes k):=(f\# g)\otimes (k\# h);
\end{equation}
notice the reversed order at the level of the second factor.
We also use the involution $^\square$ on $\mathcal{S}(\Xi\times\Xi)$ given by complex conjugation
\begin{equation}\label{jugation}
F^\square(X,Y):=\overline{F(X,Y)},\ \ \ \ \ \forall\,(X,Y)\in\Xi\times\Xi.
\end{equation}

Since $\#$ is supposed to satisfy Hypothesis A then (\ref{dublu}) will verify Hypothesis A too,
for $\Xi$ replaced by $\Xi\times\Xi$ and $\<\cdot,\cdot\>$ replaced by the duality $\<\!\<\cdot,\cdot\>\!\>$ given by
\begin{equation}\label{mutu}
\<\!\<F,G\>\!\>:=\int_\Xi\int_\Xi dX dY F(X,Y)G(X,Y).
\end{equation}

Independently on any previously defined product $\#$, on functions $:\Xi\times\Xi\rightarrow\mathbb C$ we also use
{\it the crossed product composition}
\begin{equation}\label{crosu}
(F\diamond G)(X,Y):=\int_\Xi dZ\,F(X,Z)\,G(X-Z,Y-Z)
\end{equation}
and the involution $F^\diamond(X,Y):=\overline{F(X-Y,-Y)}$.
The "crossed product" feature is seen if we consider functions $\Xi\rightarrow\mathcal{S}(\Xi)$ using notations as
$[F(Z)](X):=F(X,Z)$ and write
(\ref{crosu}) as
$$
(F\diamond G)(Y):=\int_\Xi dZ\,F(Z)\,\mathcal T_Z\left[G(Y-Z)\right]
$$
and the involution as $F^\diamond(Y)=\overline{\mathcal T_Y[F(-Y)]}$, with the action of $\Xi$ on itself given by $\mathcal T_Z(X):=X+Z$,
transferred to functions by
$\mathcal T_Z(g):=g\circ\mathcal T_{-Z}$. We refer to \cite{W} for general information on crossed product algebras.

It is well-known and follows from straightforward arguments that  $\mathcal{S}(\Xi\times\Xi)$ is a
$^*$-algebra with the structure indicated above
and that $\diamond$ is a separately continuous map on $\mathcal{S}(\Xi\times\Xi)\times\mathcal{S}(\Xi\times\Xi)$.
Let us consider the duality
$$
\<\!\<\cdot,\cdot\>\!\>':\mathcal{S}(\Xi\times\Xi)\times\mathcal{S}(\Xi\times\Xi)\rightarrow\C,\ \ \ \ \
\<\!\<F,G\>\!\>':=\int_\Xi\int_\Xi dXdYF(X,X-Y)\,G(Y,Y-X).
$$
One checks easily that
$$
\<\!\<F\diamond G,H\>\!\>'=\<\!\<F,G\diamond H\>\!\>',\ \ \ \ \ \forall\,F,G,H\in\mathcal{S}(\Xi\times\Xi),
$$
and that Hypothesis A is satisfied with respect to the above explicit form of the duality.

For completeness we recall the kernel multiplication
$$
(K\tilde\diamond L)(X,Y):=\int_\Xi dZ\,K(X,Z)\,L(Z,Y)
$$
and the involution $K^{\tilde \diamond}(X,Y):=\overline{K(Y,X)}$, that transform into the above "crossed product" operations by
the change of variables $(X,Y)\mapsto(X,X-Y)$.

The same extensions by duality as those given in subsection \ref{pozution} also work for the composition laws $\square,\,\diamond$
and $\tilde\diamond$ and they will be used below.

\section{Modulation mappings}\label{modmap}

\subsection{The canonical mappings}

Assuming Hypothesis A, we start with $\mathbf R:\mathcal{S}(\Xi\times\Xi)\rightarrow\mathcal{S}'(\Xi\times\Xi)$, uniquely defined by
\begin{equation}\label{inceputul}
\<\!\<\R(f_1\otimes f_2),g_1\otimes g_2\>\!\>=\<g_1\#f_1,g_2\#f_2\>,
\ \ \ \ \ \forall\,f_1,f_2,g_1,g_2\in\mathcal{S}(\Xi)\,.
\end{equation}
This definition is justified by the identifications (\ref{tens1}), the universal properties of the topological tensor products
and the fact that the duality $\<\cdot,\cdot\>$ and of the product $\#$ are bilinear and continuous.
To arrive to a convenient setting, we are also going to need Hypothesis B.

\begin{Proposition}\label{def-R}
Under Hypothesis A and B the above defined map $\mathbf R$ is in fact a linear continuous map of the Schwartz space $\mathcal{S}(\Xi\times\Xi)$
into itself that extends uniquely to a linear continuous map of the space of temperate distributions $\mathcal{S}'(\Xi\times\Xi)$
into itself.
\end{Proposition}

\begin{proof}
For $f_1,f_2,g_1,g_2\in\mathcal{S}(\Xi)$, using the cyclicity property following from Hypothesis A, we can write
$$
\<\!\<\mathbf R(f_1\otimes f_2),g_1\otimes g_2\>\!\>=\<g_1,f_1\#g_2\#f_2\>.
$$
By Hypothesis B, this can be extended to the case $g_1,g_2\in\mathcal{S}'(\Xi)$. The necessary continuity properties needed to justify
$\R$ as a linear continuous mapping from $\S(\Xi\times\Xi)$ to the strong dual of $\S'(\Xi\times\Xi)$
(on which we consider the strong topology)
follows basically from Remark \ref{totusi}. Then we use the reflexivity of $\S(\Xi\times\Xi)$ to identify it with this dual.

A simple calculation shows that $\R$ coincides with its adjoint $\R^*$ :
$$
\begin{aligned}
&\<\!\<\overline{\R(f_1\otimes f_2)},g_1\otimes g_2\>\!\>=\overline{\<\!\<\R(f_1\otimes f_2),\overline g_1\otimes
\overline g_2\>\!\>}=\overline{\<\overline g_1\#f_1,\overline g_2\#f_2\>}\\
&=\<\overline f_1\#g_1,\overline f_2\#g_2\>=\<\!\<\R(g_1\otimes g_2),\overline f_1\otimes \overline f_2\>\!\>
=\<\!\<\overline{f_1\otimes f_2},\R(g_1\otimes g_2)\>\!\>.
\end{aligned}
$$
In terms of the transpose this is written
$\R={\kappa}\circ \R^T\circ{\kappa}$ with $\kappa$ the complex conjugation.
This formula allows us to extend $\R$ to a linear continuous map on $\mathcal{S}'(\Xi\times\Xi)$.
\end{proof}

\begin{Corollary}
Under Hypothesis A and B, for any two test functions $f$ and $g$ in $\mathcal{S}(\Xi\times\Xi)$ we have the formula
$$
\big[\R(f\otimes g)\big](X,Y)\,=\,\<\delta_X\#f,\delta_Y\#g\>.
$$
\end{Corollary}

\begin{proof}
We use the extension of $\R(f\otimes g)$ to a linear functional on $\S'(\Xi\times\Xi)$ to compute
$$
\big[\R(f\otimes g)\big](X,Y)=\<\!\<\R(f\otimes g),\delta_X\otimes \delta_Y\>\!\>\,=\,\<\delta_X\#f,\delta_Y\#g\>.
$$
\end{proof}

Let us introduce now two slightly transformed versions of our map $\R$.

\begin{Definition}
Suppose given a composition $\#:\mathcal{S}(\Xi)\times\mathcal{S}(\Xi)\rightarrow\S(\Xi)$ satisfying Hypothesis A. We define:
\begin{itemize}
\item {\it The $\N$-function}, uniquely determined by
\begin{equation}\label{N}
\<\!\<\N(f_1\otimes f_2),g_1\otimes g_2\>\!\>:=\<\widehat{g}_1\#f_1,\widehat{g}_2\#f_2\>.
\end{equation}
\item {\it The modulation mapping}.
We consider the change of variables map on $\Xi\times\Xi$
\begin{equation}\label{C}
C:\Xi\times\Xi\rightarrow\Xi\times\Xi,\ \ C(X,Y):=(-X,X-Y),
\end{equation}
\begin{equation}\label{CC}
\quad \mathbf C:\mathcal{S}(\Xi\times\Xi)\rightarrow\mathcal{S}(\Xi\times\Xi),\ \ \mathbf C(F):=F\circ C
\end{equation}
and define
\begin{equation}\label{M}
\M:=\N\circ\mathbf C.
\end{equation}
\end{itemize}
\end{Definition}

\begin{Remark}\label{relations}
{\rm The three linear maps $\R,\N,\M$ verify the relations:
\begin{enumerate}
\item $\R=\left(\mathfrak{F}\otimes\mathfrak{F}\right)\circ \N$.
\item $[\N(f\otimes g)](X,Y)=\<\e_X\#f,\e_Y\#g\>$.
\item $ [\M(f\otimes g)](X,Y)=\<\e_{-X}\#f,\e_{X-Y}\#g\>$.
\end{enumerate}
}
\end{Remark}

In absence of an explicit form of the multiplication $\#$, it is not easy to write down
$\M(F)$ for $F\in \mathcal{S}(\Xi\times\Xi)$ not being an elementary vector of the form $f_1\otimes f_2$.

\begin{Theorem}\label{prigoana}
Under Hypothesis A and B, the mappings $\M,\N,\R$ define linear continuous mappings of the Schwartz space $\mathcal{S}(\Xi\times\Xi)$.
They extend uniquely to linear continuous mappings of the space of temperate distributions $\mathcal{S}'(\Xi\times\Xi)$.
\end{Theorem}

\begin{proof}
We use Proposition \ref{def-R} and the Remark \ref{relations} together with the well known continuity properties of the Fourier
transform and of the change of variables transformation.
\end{proof}

\subsection{The $L^2$ extension}

Hypothesis A and B are not enough to insure isomorphism properties (or at least
non-triviality) for the application $\M$. We have succeeded to isolate an extra condition leading to a perfect behavior of
the mappings $\M,\N,\R$, which is fulfilled in the applications we have in mind.

Let us consider the following $\#$-induced action of $\Xi$ on $\mathcal{S}'(\Xi)$:
\begin{equation}\label{tracas}
\forall Z\,\in\Xi,\quad\Theta^\#_Z:\mathcal{S}'(\Xi)\rightarrow\mathcal{S}'(\Xi),\ \ \ \ \ \Theta^\#_Z(f):=\e_{-Z}\,\# f\,\#\e_{Z}.
\end{equation}
Notice that they restrict to automorphisms of the $^*$-algebras $\mathcal{S}(\Xi)$ and $L^2(\Xi)$.
The correspondence $Z\mapsto \Theta^\#_Z(f)$ might not have remarkable group properties but we require at least

\medskip
\noindent
{\bf Hypothesis C.} For any $f,g\in\mathcal{S}(\Xi)$
\begin{equation}\label{ceoare}
\int_\Xi\int_\Xi dYdZ\left[\Theta_Z(f)\right](Y)\,g(Y)=\int_\Xi dZ f(Z)\int_\Xi dY g(Y).
\end{equation}
Clearly a commutative product $\#$ cannot satisfy (\ref{ceoare}); the l.h.s. would not even be defined.

\begin{Theorem}\label{incep}
If the composition $\#:\mathcal{S}(\Xi)\times\mathcal{S}(\Xi)\rightarrow\S(\Xi)$ satisfies Hypothesis A, B and C,
then all the mappings $\M,\N,\R$ extend to unitary operators on the Hilbert space $L^2(\Xi\times\Xi)$.
\end{Theorem}

\begin{proof}
Taking into account the relations between the three mappings, it is enough to make the proof for one of them.
We are going to work with $\M$, the main object subsequently.

For $f_1,g_1,f_2,g_2\in \mathcal{S}(\Xi)$, using (\ref{croco}), (\ref{pleasca}) and (\ref{ceoare}) we compute
$$
\begin{aligned}
\<\!\<\overline{\M(f_1\otimes f_2)},\M(g_1\otimes g_2)\>\!\>&=\int_\Xi\int_\Xi dXdY\<\overline{\e_{-X}\# f_1\#\e_{X-Y}},
\overline f_2\>\<\e_{-X}\#g_1\#\e_{X-Y},g_2\>\\
&=\int_\Xi\int_\Xi dXdY\<\e_{Y-X}\#\overline f_1\#\e_{X},\overline f_2\>\<\e_{-X}\#g_1\#\e_{X-Y},g_2\>\\
&=\int_\Xi\int_\Xi dXdY \< \overline f_1\#\e_{X}\#\overline f_2,\e_{Y-X}\>\<\e_{X-Y},g_2\#\e_{-X}\#g_1\>\\
&=\int_\Xi dX\<\overline f_1\#\e_X\# \overline f_2,g_2\#\e_{-X}\#g_1\>\\
&=\int_\Xi dX\<\e_X\# \overline f_2\#g_2,\e_{-X}\#g_1\#\overline f_1\>\\
&=\int_\Xi dX\<\Theta_X(\overline f_2\#g_2),g_1\#\overline f_1\>\\
&=\int_\Xi dX\,(\overline f_2\#g_2)(X)\int_\Xi dY\,(g_1\#\overline f_1)(Y)\\
&=\<\overline f_2,g_2\>\<\overline f_1,g_1\>
=\<\!\<\overline{f_1\otimes f_2},g_1\otimes g_2\>\!\>.
\end{aligned}
$$
Therefore, one can extend $\M$ to an isometry of $L^2(\Xi\times\Xi)$.
However, seen again as an operator on $\mathcal{S}(\Xi\times\Xi)$, one has
$$
\M^*=\left[(\mathfrak F\otimes\mathfrak F)\circ \R\circ \mathbf C\,\right]^*=\mathbf C\circ \R\circ(\mathfrak F\otimes\mathfrak F)=
\mathbf C\circ (\mathfrak F\otimes\mathfrak F)\circ \M\circ \mathbf C\circ(\mathfrak F\otimes\mathfrak F),
$$
so $\M^*$ also extends to an isometry of $L^2(\Xi\times\Xi)$. This is enough to conclude.
\end{proof}

\begin{Corollary}\label{rolar}
If Hypothesis C is also fulfilled, the mappings $\M,\N,\R$ are topological isomorphisms of $\mathcal{S}(\Xi\times\Xi)$,
extending to topological isomorphisms of $\mathcal{S}'(\Xi\times\Xi)$.
\end{Corollary}

\begin{proof}
Follows from Theorems \ref{prigoana} and \ref{incep}.
\end{proof}

\subsection{Algebraic properties}\label{gebraic}

\begin{Theorem}\label{fundamentala}
Under Hypothesis A and B, the application
\begin{equation}\label{udatu}
\M:\(\mathcal{S}(\Xi\times\Xi),\square,^\square\)\rightarrow\(\mathcal{S}(\Xi\times\Xi),\diamond,^\diamond\)
\end{equation}
is a morphism of $^*$-algebras. If Hypothesis C is also verified, it is an isomorphism.
\end{Theorem}

\begin{proof}
We are going to use (\ref{croco}) and (\ref{pleasca}) to show that
\begin{equation}\label{minunata}
\M(f_1\otimes f_2)\diamond \M(g_1\otimes g_2)=\M[(f_1\otimes f_2)\square(g_1\otimes g_2)]
\end{equation}
for any $f_1,f_2,g_1,g_2\in\S(\Xi)$; this is enough to prove that $\M$ intertwines the products $\square$ and $\diamond$.
One has for all $X,Y\in\Xi$
$$
\begin{aligned}
\left[\M(f_1\otimes f_2)\diamond \M(g_1\otimes g_2)\right](X,Y)&=\int_\Xi dZ[\M(f_1\otimes f_2)](X,Z)[\M(g_1\otimes g_2)](X-Z,Y-Z)\\
&=\int_\Xi dZ\<\e_{-X}\#f_1\#\e_{X-Z},f_2\>\<\e_{Z-X}\#g_1\#\e_{X-Y},g_2\>\\
&=\int_\Xi dZ\<f_2\#\e_{-X}\#f_1,\e_{X-Z}\>\<\e_{Z-X},g_1\#\e_{X-Y}\#g_2\>\\
&=\<f_2\#\e_{-X}\#f_1,g_1\#\e_{X-Y}\#g_2\>\\
&=\<\e_{-X}\#(f_1\#g_1)\#\e_{X-Y},g_2\# f_2\>\\
&=\M[(f_1\#g_1)\otimes(g_2\,\#f_2)](X,Y)\\
&=\M[(f_1\otimes f_2)\square(g_1\otimes g_2)](X,Y).
\end{aligned}
$$
For the involution:
$$
\begin{aligned}
\left[\M(\overline f_1\otimes \overline f_2)\right](X,Y)&=\<\e_{-X}\#\overline f_1\#\e_{X-Y},\overline f_2\>
=\overline{\<\e_{Y-X}\#f_1\#\e_X,f_2\>}\\
&=\overline{[\M(f_1\otimes f_2)](X-Y,-Y)}
=[\M(f_1\otimes f_2)]^\diamond (X,Y).
\end{aligned}
$$
The last part of the statement follows from Theorem \ref{incep}.
\end{proof}

The same can be stated about
\begin{equation}\label{udatu}
\R:\(\mathcal{S}(\Xi\times\Xi),\square,^\square\)\rightarrow\(\mathcal{S}(\Xi\times\Xi),\tilde\diamond,^{\tilde\diamond}\)
\end{equation}
and
\begin{equation}\label{udatu}
\N:\(\mathcal{S}(\Xi\times\Xi),\square,^\square\)\rightarrow\(\mathcal{S}(\Xi\times\Xi),\tilde\diamond,^{\tilde\diamond}\).
\end{equation}
The proofs are computations similar to those above; one can also use the relations between $\M,\N,\R$ contained in Remark \ref{relations}
and the algebraic properties of the transformations $\mathfrak F$ and $\mathbf C$.

\begin{Remark}\label{ben}
{\rm Let us recall the well known fact that kernel composition leaves $L^2(\Xi\times\Xi)$
invariant being separately continuous with respect to the $\|\cdot\|_{L^2}$ norm (that we shall simply denote by $\|\cdot\|_{2}$).
Actually $\(L^2(\Xi\times\Xi),\tilde\diamond,^{\tilde\diamond}\)$ is a normed $^*$-algebra containing $\S(\Xi\times\Xi)$ densely.
Then, obviously, $\(L^2(\Xi\times\Xi),\square,^\square\)$ will also be a normed $^*$-algebra and the same can be said about
$\Big(L^2(\Xi),\#,^\#\Big)$.
}
\end{Remark}

\subsection{Localization}\label{lizeisn}

For any $h\in\mathcal{S}'(\Xi)$ we define the linear injective map
\begin{equation}\label{provicia}
J_h:\mathcal{S}'(\Xi)\rightarrow\S'(\Xi)\widehat\otimes\S'(\Xi)\cong\mathcal{S}'(\Xi\times\Xi),\ \ \ \ \ J_h(f):=f\otimes h,
\end{equation}
which clearly restricts to a linear injection $J_h:\mathcal{S}(\Xi)\rightarrow\mathcal{S}(\Xi)\widehat\otimes\mathcal{S}'(\Xi)$.
It can be shown easily that $J_h$ is continuous when considering on $\mathcal{S}'$ either the weak or the strong topology, respectively.

Although a general $h$ can be useful, and we think of the case
$h=1$ for instance, in the present article we are only going to consider $h\in\S(\Xi)$.
Actually, if $h\in\mathcal{S}(\Xi)$, one also has $J_h:\mathcal{S}(\Xi)\rightarrow\mathcal{S}(\Xi\times\Xi)$, $J_{h}:L^2(\Xi)\rightarrow
L^2(\Xi\times\Xi)$ and $J_h:\mathcal{S}'(\Xi)\rightarrow\mathcal{S}'(\Xi)\widehat\otimes\mathcal{S}(\Xi)$ as linear continuous injections.
The transpose map $\widetilde J_h:\S'(\Xi\times\Xi)\rightarrow\S'(\Xi)$
\begin{equation}\label{jtilda}
\<\widetilde J_h(F),f\>:=\<\!\<F,J_h(f)\>\!\>,\ \quad F\in\S'(\Xi\times\Xi),\ f\in\S(\Xi)
\end{equation}
will also be useful. We record the relations valid for $h,k\in\S(\Xi)$:
\begin{equation}\label{vortex}
\widetilde J_k J_h=\<k,h\>{\rm id},\ \ \ \ \ J_h\widetilde J_k={\rm id}\otimes|h\rangle\langle k|.
\end{equation}

Let us introduce now the main tool:

\begin{Definition}\label{ools}
Let $h\in \S(\Xi)\setminus\{0\}$ be given (we might call it {\it a window}). We define $\M_h:\S'(\Xi)\rightarrow\S'(\Xi\times\Xi)$ and
$\widetilde \M_h:\S'(\Xi\times\Xi)\rightarrow\S'(\Xi)$ by
\begin{equation}\label{cordingly}
\M_h:=\M\circ J_h,\ \ \ \ \ \M_h(f):=\M(f\otimes h),
\end{equation}
and
\begin{equation}\label{cardingly}
\widetilde \M_h:=\widetilde J_h\circ \M^{-1}.
\end{equation}
\end{Definition}

We have the following equalities:
\begin{equation}\label{vertex}
\widetilde \M_k \M_h=\<k,h\>{\rm id},\ \ \ \ \ \M_h\widetilde \M_k=\M({\rm id}\otimes|h\rangle\langle k|)\M^{-1}
\end{equation}
and
\begin{equation}\label{adjunctu}
\widetilde \M_{h}(G)=\int_\Xi \int_\Xi dXdY\,G(X,Y)\,\e_X\# {h}\# \e_{Y-X}.
\end{equation}
In particular, $\frac{1}{\parallel h\parallel}\M_{h}$ is an $L^2$-isometry. If $\<k,h\>\ne 0$,
one could call the first relation in (\ref{vertex}) {\it the inversion formula}.

Taking into consideration the algebraic properties of the isomorphism $\M$ and
$$
J_h(f)\square J_k(g)=J_{k\# h}(f\#g),\ \quad J_h(f)^\square=J_{h\!^\#}(f^\#),
$$
one gets

\begin{Corollary}\label{fundamentalu}
If $\,h\# h=h=\overline h\ne 0$, then $\M_{h}:\left(L^2(\Xi),\#,^\#\right)\rightarrow
\left(L^2(\Xi\times\Xi),\diamond,^\diamond\right)$ is an injective morphism of $^*$-algebras, which sends the $^*$-subalgebra
$\S(\Xi)$ into $\S(\Xi\times\Xi)$.
\end{Corollary}

\section{Modulation spaces of symbols}\label{spasy}

\subsection{General facts}\label{spasoy}

We shall always assume that Hypothesis A, B and C hold.
If $0\ne h\in\mathcal{S}(\Xi)$, then $\M_h:\mathcal{S}'(\Xi)\rightarrow
\mathcal{S}'(\Xi\times\Xi)$ is a linear continuous injection. We can use it to pull-back structure from the final space.

\begin{Definition}\label{induc}
Let $\mathfrak L$ be a vector subspace of $\mathcal{S}'(\Xi\times\Xi)$. The vector subspace
$\M_{h}^{-1}(\mathfrak L)$ of $\mathcal{S}'(\Xi)$ is denoted by $\mathfrak L^\M_h$.
\end{Definition}

If it is granted that we only work with $\M$ (and not with $\R$ or $\N$), another good notation would be $\LL^\#_h$.
On the other hand, taking into account the simple connection between the isomorphisms $\M,\N$ and $\R$, we see that any one could be used
to induce spaces in equivalent ways. Using $\M$ is just a matter of convenience. However, in a more general setting
(in which the Fourier transformation is no longer available), using $\R$ might be the single choice.

\begin{Remark}\label{prostii}
{\rm If $h\in\mathcal{S}(\Xi)\setminus\{0\}$, then obviously $f\in\mathcal{S}(\Xi)$ if and only if $J_h(f)\in\mathcal{S}(\Xi\times\Xi)$.
This implies that $f\in\mathcal{S}(\Xi)$ if and only if $\M_h(f)\in\mathcal{S}(\Xi\times\Xi)$.
Thus $\mathcal{S}(\Xi\times\Xi)^\M_h=\mathcal{S}(\Xi)$, so
$\mathcal{S}(\Xi\times\Xi)\subset\LL$ implies $\mathcal{S}(\Xi)\subset\LL^\M_h$.
It is also clear that $L^2(\Xi\times\Xi)^\M_h=L^2(\Xi)$ and $\mathcal{S}'(\Xi\times\Xi)^\M_h=\mathcal{S}'(\Xi)$.}
\end{Remark}

As in many references on coorbit theory and modulation spaces (\cite{F2,FG,G1,T2} and many others) we start by establishing
the nature of the spaces produced by Definition \ref{induc}.

\begin{Proposition}\label{sistom}
Let $\big(\LL,\parallel\cdot\parallel_\LL\big)$ be a normed space continuously embedded in $\mathcal{S}'(\Xi\times\Xi)$.
\begin{enumerate}
\item For $h\in\mathcal{S}(\Xi)\setminus\{0\}$, $\mathfrak L^\M_h:=\M_{h}^{-1}(\mathfrak L)$ is a normed space
with $\parallel\cdot\parallel_{\mathfrak L^\M_h}:=\parallel \M_h\cdot\parallel_\mathfrak L$\,.
\item
If $\LL$ is a Banach space, $\LL^\M_h$ is also a Banach space.
\end{enumerate}
\end{Proposition}

\begin{proof}
The first assertion is obvious, so we only need to show that the inducing process preserves completeness.

Clearly $\mathfrak L^\M_h$ is isometrically isomorphic to $\mathfrak L(\M,h):=\mathfrak L\,\cap\,\M_h[\mathcal{S}'(\Xi)]$.
So we have to show that $\mathfrak L(\M,h)$ is closed in $\mathfrak L$.
Let $\big(\M_h(f_n)\big)_{n\in\N}$ be a sequence in $\mathfrak L(\M,h)$, converging to $G\in\mathfrak L$.
Due to the continuity of the embedding of $\LL$ in $\mathcal{S}'(\Xi\times\Xi)$ we also have that
\begin{equation}\label{acciu}
\M_h(f_n)\underset{n\rightarrow\infty}{\longrightarrow}G\ \ {\rm in}\ \mathcal{S}'(\Xi\times\Xi).
\end{equation}
Set $f:=\frac{1}{\parallel h\parallel^2}\widetilde{\M}_h(G)$. Using the inversion formula one gets
$$
f-f_n=\frac{1}{\parallel h\parallel^2}\widetilde{\M}_h[G-\M_h(f_n)]
\underset{n\rightarrow\infty}{\longrightarrow}0\ \ {\rm in}\ \ \mathcal{S}'(\Xi)
$$
which, together with (\ref{acciu}), imply $G=\M_h(f)\in\mathfrak L(\M,h)$.
\end{proof}

\begin{Example}
One can work \cite{G1} with the Banach spaces $\LL=L^{p,q}(\Xi\times\Xi)\subset\S'(\Xi\times\Xi)$,
defined for $p,q\in[1,\infty)$ by the norm
$$
\parallel F\parallel_{L^{p,q}}:=\(\int_\Xi dY\[\int_\Xi dX |F(X,Y)|^p\]^{q/p}\)^{1/q}.
$$
The cases $p=\infty$ or/and $q=\infty$ require the usual modifications; one can also introduce weights.
\end{Example}

We study now the dependence of the induced norms and spaces on the window $h$ and on the mapping $\M$.

\begin{Proposition}\label{sistem}
Assume that for some $h,k\in\mathcal{S}(\Xi)\setminus\{\,0\}$ one has
\begin{equation}\label{lumitun}
\big(\M_k \widetilde{\M}_h\big)\mathfrak L\subset\mathfrak L.
\end{equation}
Then $\mathfrak L^\M_h\subset\mathfrak L^\M_k$. In addition, if $\LL$ is a Banach space continuously embedded in
$\mathcal{S}'(\Xi\times\Xi)$, the embedding of $\mathfrak L^\M_h$ in $\mathfrak L^\M_k$ is continuous.
\end{Proposition}

\begin{proof}
By the inversion formula,
if $f\in\LL^\M_h$, then
$$
f=\frac{1}{\parallel h\parallel^2}\widetilde \M_h[\M_h(f)]\in \widetilde \M_h(\LL),
$$
which implies $\mathfrak L^\M_h\subset \widetilde \M_h(\mathfrak L)$.
So we need to show that $\widetilde \M_h(\mathfrak L)\subset\mathfrak L^\M_k$. But
$$
f=\widetilde \M_h(G),\ \,{\rm with}\ G\in\LL\ \ \Longrightarrow\ \ \M_k(f)=[\M_k \widetilde \M_h](G)\in\LL.
$$
To prove the topological embedding, note that $\M_k \widetilde \M_h\in\mathbb B(\mathfrak L)$ by the Closed Graph Theorem.
This and the inversion formula give easily the norm estimate needed.
\end{proof}

Let us say that the Banach space $\mathfrak L$ continuously embedded in $\S'(\Xi\times\Xi)$ is {\it admissible} if
one has $\big(\M_k \widetilde{\M}_h\big)\mathfrak L\subset\mathfrak L$ for every $h,k\in\mathcal{S}(\Xi)\setminus\{\,0\}$.
By the result above, if $\LL$ is admissible, we could speak of the Banachizable space
$\mathfrak L^\M:=\M_h^{-1}(\mathfrak L)$, which is continuously embedded in $\mathcal{S}'(\Xi)$;
the vector space and the topology do not depend on $h\in\S(\Xi)\setminus\{\,0\}$.

In general, (\ref{lumitun}) may fail even for $h=k$. There are many pairs $(h,\LL)$ for which
$J_h\widetilde{J}_h={\rm id}\otimes |h\rangle\langle h|$
does not leave $\M^{-1}\LL$ invariant. However, if we assume that $(\M_h\widetilde{\M}_h)\LL\subset\LL$ (implying that
$\M_h\widetilde{\M}_h\in\mathbb B(\LL))$, one gets $\widetilde{\M}_h\in\mathbb B(\LL,\LL^M_h)$. If, in addition, $\S(\Xi\times\Xi)$
is contained densely in $\LL$, it is a simple exercise to prove that $\S(\Xi)$ is dense in $\LL^M_h$. This happens for all $h$
if $\LL$ is admissible and contains $\S(\Xi\times\Xi)$ densely.

\medskip
Assume now that we are given a second topological linear isomorphism $\mathbb M:\S'(\Xi\times\Xi)\rightarrow\S'(\Xi\times\Xi)$;
it might not have other remarkable algebraic properties. For windows $0\ne h\in\S(\Xi)$, one can construct in the same way as before
injective linear continuous maps $\mathbb M_h(\cdot):=\mathbb M\,(\cdot\otimes h):\S'(\Xi)\rightarrow\S'(\Xi\times\Xi)$.
Consequently, if $\LL$ is a Banach space continuously included in $\S'(\Xi\times\Xi)$, one constructs analogously the norms
$\parallel\cdot\parallel_{\LL^\mathbb M_h}:=\parallel \mathbb M_h(\cdot)\parallel_\LL$ and the Banach spaces $\LL^\mathbb M_h$.
The following result is easy to prove (use the Closed Graph Theorem) and shows when we are going to obtain the same spaces as before.

\begin{Proposition}\label{etp}
Assume that the isomorphism $\M\circ\mathbb M^{-1}$ of $\S'(\Xi\times\Xi)$ restricts to a bijection $\LL\rightarrow\LL$.
Then the two subspaces $\LL^\mathbb M_h$ and $\LL^\mathbf M_h$ of $\S'(\Xi)$ coincide and have equivalent Banach norms
for every non-trivial window $h$.
\end{Proposition}

This could be useful in certain concrete cases to make the connection with some different, more traditional approach. While our
$\M$ is defined very generally and it is a $^*$-algebraic isomorphism, in examples other equivalent choices might have other good properties
or offer a better intuition.

\subsection{Algebras}\label{spasiy}

\begin{Definition}\label{loyal}
We introduce the following subspace of tempered distributions on $\Xi$
$$
\mathcal S'_\#(\Xi):=\left\{f\in\mathcal S'(\Xi)\ \mid f\#\mathcal S(\Xi)\subset\mathcal S(\Xi),\ \,\mathcal S(\Xi)\#f\subset\S(\Xi)\right\}
$$
and call it {\it the  Moyal algebra associated to $\#$}.
\end{Definition}

Even in the standard examples it is quite difficult to find a direct description of the set $\mathcal S'_\#(\Xi)$. It is expected to
be quite large, as checked in examples in \cite{GBV,Ho,MP1}.

Obviously $\S'_\#(\Xi)$ is invariant under the involution $f\mapsto f^\#:=\overline f$. We extend the composition
$\#:\mathcal S'_\#(\Xi)\times\mathcal S'_\#(\Xi)\rightarrow\mathcal S'_\#(\Xi)$ by setting
\begin{equation}\label{extind}
\<f\#g,h\>:=\<f,g\#h\>,\ \quad \forall\,f,g\in\mathcal S'_\#(\Xi),\,h\in\S(\Xi).
\end{equation}
To justify this extension, notice first that if $g\in\S'_\#(\Xi)$, then the mapping $\S(\Xi)\ni h\mapsto L_g(h)=g\#h\in\S(\Xi)$ is
well defined, linear and continuous. Actually it is continuous when regarded with values in $\S'(\Xi)$, by Corollary \ref{ciuhap},
and then we use the Closed Graph Theorem to get the improved continuity. It follows (even for $f\in\S'(\Xi)$) that the formula
(\ref{extind}) defines a tempered distribution $f\#g$. To prove that in fact it belongs to the Moyal algebra, one must show
that $(f\# g)\#k\in\S(\Xi)$ and $k\#(f\# g)\in\S(\Xi)$ for any $k\in\S(\Xi)$. For the first one, for instance, one shows by approximation
from the previously results that $(f\# g)\#k=f\#(g\#k)$ and then apply the definition of $\S'_\#(\Xi)$.

Without making all the verification, we just state that $\(\S'_\#(\Xi),\#,^\#\)$ {\it is a $^*$-algebra in which $\S(\Xi)$
is a self-adjoint two-sided ideal}. More details about Moyal algebras associated to abstract composition laws and their connections with
algebraic representations in Hilbert and other locally convex spaces will be given in a subsequent publication.

The same procedure can be applied to the $^*$-algebra $\(\S(\Xi\times\Xi),\square,^\square\)$, getting a Moyal algebra
$\S'_\square(\Xi\times\Xi)$, and to $\(\S(\Xi\times\Xi),\diamond,^\diamond\)$, getting a Moyal algebra
$\S'_\diamond(\Xi\times\Xi)$.

\begin{Corollary}\label{fundamentala}
\begin{enumerate}
\item
The application $\M$ restricts to an isomorphism of $^*$-algebras
\begin{equation}\label{udatu}
\M:\mathcal{S}'_\square(\Xi\times\Xi)\rightarrow\mathcal{S}'_\diamond(\Xi\times\Xi).
\end{equation}
\item
If $\,h\#h=\overline h=h\in\S(\Xi)\setminus\{0\}$, the application $\M_h$ restricts to an injective morphism of $^*$-algebras
$\M_h:\mathcal{S}'_\#(\Xi)\rightarrow\mathcal{S}'_\diamond(\Xi\times\Xi)$
and one has $\mathcal{S}'_\diamond(\Xi\times\Xi)^\M_h =\mathcal{S}'_\#(\Xi)$.
\item
If $\LL$ is a $^*$-subalgebra of $\mathcal{S}'_\diamond(\Xi\times\Xi)$, then $\LL^\M_h$ is a $^*$-subalgebra of $\mathcal{S}'_\#(\Xi)$.
\end{enumerate}
\end{Corollary}

\begin{proof}
This follows easily from the previous results.
\end{proof}

Although convenient and rather general, the formalism involving Moyal algebras does not cover all the interesting examples.
For instance, the mixed Lebesgue space $L^{\infty,1}(\Xi\times\Xi)$ is a Banach $^*$-algebra,
$\S(\Xi\times\Xi)$ is contained in it, but it is not an ideal.

\begin{Remark}\label{arly}
{\rm It is easy to prove more general composition results. For example, if $\mathfrak G,\mathfrak K,\mathfrak L$ are
subspaces of $\S'(\Xi\times\Xi)$ and $\mathfrak G\diamond\mathfrak K\subset\mathfrak L$, for every $h=h\# h$ one gets
$\mathfrak G^\M_h\#\mathfrak K^\M_h\subset\mathfrak L^M_h$.}
\end{Remark}

\begin{Remark}\label{harl}
{\rm Obviously, if $\parallel\cdot\parallel_\LL$ is a $C^*$-norm then $\parallel\cdot \parallel_{\LL^\M_h}$
will also be a $C^*$-norm:
$$
\parallel f^\#\# f\parallel_{\mathfrak L^\M_h}\,=\,\parallel \M_h(f^\#\# f)\parallel_{\mathfrak L}\,=\,
\parallel \M_h(f)^\diamond\diamond \M_h(f)\parallel_{\mathfrak L}\,=\,\parallel \M_h(f)\parallel_{\mathfrak L}^2\,=\,
\parallel f\parallel^2_{\mathfrak L^\M_h}.
$$
}
\end{Remark}

\subsection{Spectral invariance}\label{spachy}

We address now the problem of {\it invariance under inversion}. If $\mathfrak K$ is a unital subalgebra of an algebra $\LL$
and for any element $F\in\mathfrak K$ invertible in $\LL$ one has $F^{-1}\in\mathfrak K$, we say that $\mathfrak K$ is
{\it spectrally invariant} in $\LL$. One also uses terms as {\it Wiener subalgebra, subalgebra invariant under inversion}, etc.

\begin{Proposition}\label{wiener}
Assume that unital algebras $\mathfrak K,\mathfrak L$ are given such that
$\mathfrak K\subset\LL\subset\mathcal{S}'_\diamond(\Xi\times\Xi)$ and $\mathfrak K$ is spectrally invariant in $\LL$.
If $\mathfrak K^\M_h$ and $\LL^\M_h$ are unital (algebras), then $\mathfrak K^\M_h$ is spectrally invariant in $\LL^\M_h$ for every
$h=h\#h\in\mathcal{S}(\Xi)\setminus\{\,0\}$.
\end{Proposition}

\begin{proof}
Clearly we have $ \mathfrak K^\M_h\subset\LL^\M_h\subset\mathcal{S}'_\#(\Xi)$.

Let us consider the spaces $\LL(\M,h):=\LL\cap \M_h[\S'(\Xi)]$ and $\mathfrak K(\M,h):=\mathfrak K\cap \M_h[\S'(\Xi)]$
that are subalgebras of
$\LL$ and resp. of $\mathfrak K$. Since by hypothesis $\LL^\M_h$ and $\mathfrak K^\M_h$ are unital, it follows that $H:=\M_h(1)$
is a projection that belongs to $\mathfrak K$; then $\mathbf 1-H=:H^\bot$ also belongs to this algebra.
Let us notice that
\begin{equation}\label{gandeste}
\forall\, F\in\LL(\M,h),\quad H\diamond F = F\diamond H = F.
\end{equation}
Let $f\in\mathfrak K^\M_h\subset\LL^\M_h$ that has an inverse $g\in\LL^\M_h$;
thus $f\#g=g\# f=1$ in $\LL^\M_h$. We conclude that $\M_h(f)\in\mathfrak K$,\, $\M_h(g)\in\LL$ and
$\M_h(f)\diamond \M_h(g)=\M_h(g)\diamond \M_h(f)=H$. By (\ref{gandeste}), one has in $\LL$
$$
\big(\M_h(f)+H^\bot\big)\diamond\big(\M_h(g)+H^\bot\big)=H+\M_h(f)\diamond H^\bot+H^\bot\diamond \M_h(g)+H^\bot=H+H^\bot=\mathbf 1
$$
and similarly
$$
\big(\M_h(g)+H^\bot\big)\diamond\big(\M_h(f)+H^\bot\big) = \mathbf 1.
$$
Thus $\M_h(f)+H^\bot$ belongs to $\mathfrak K$ and has an inverse $\M_h(g)+H^\bot$ in $\LL$; due to the spectral invariance
hypothesis we conclude that $\M_h(g)+H^\bot$ belongs in fact to $\mathfrak K$. But $H^\bot\in\mathfrak K$,
so we conclude that $\M_h(g)\in\mathfrak K$ and consequently $g\in\mathfrak K^\M_h$.
\end{proof}

The framework above is somehow artificial. One reason is that in many cases $\LL^\M_h$ is unital, but $\mathfrak K_h^\M$ isn't.
We can still work by requiring that the unitization of $\mathfrak K_h^\M$ should be spectrally invariant in $\mathfrak L_h^\M$.
On the other hand, even the choice of a space $\LL$ defining $\mathfrak L_h^\M$ by the inducing process can be seen as an artefact
of the present intrinsic setting. If one disposes of a faithful representation of $\mathfrak K_h^\M$ in some Hilbert space,
maybe obtained as the restriction of a faithful representation of the $^*$-algebra $\S'_\#(\Xi)$ by linear continuous operators
acting between suitable locally convex spaces, the problem of spectral invariance can be state in more concrete terms.
So we postpone its deeper study to a future work.

\section{Examples}\label{lizeisa}

\subsection{The point-wise product}\label{pwp}
Clearly, $\SS(\Xi)\times\SS(\Xi)\ni(f,g)\mapsto fg\in\SS(\Xi)$ satisfies Hypothesis A. The Moyal algebra $\S'_\cdot(\Xi)$ coincides
with $C^\infty_{{\rm pol}}(\Xi)$, the space of smooth functions with polynomially bounded derivatives. The modulation mapping is given by
$$
\M(f\otimes g)=1\otimes\mathfrak F(fg)=1\otimes\(\widehat f\ast\widehat g\),\ \ \ \ \ \forall\,f,g\in\S(\Xi),
$$
or more generally
$$
\M(F)=1\otimes \widehat{F}_\Delta,\ \ \ \ \ F_\Delta(X):=F(X,X),\ \ \ \forall\,F\in\S(\Xi\times\Xi),\ X\in\Xi.
$$
Plainly, $\M$ does not send $\S(\Xi\times\Xi)$ into itself; neither is it injective or surjective in some reasonable sense.
There are no Orthogonality Relations and Hypothesis C fails. Hypothesis B fails too.

\subsection{The standard Weyl calculus}\label{lizeisb}

It has as a background the problem of quantization of a physical system consisting in a spin-less particle moving in the
euclidean space $\X:=\mathbb R^n$. I recall that the phase space is $\Xi:=T^*\X\equiv\X\times\X'$, containing points
$X=(x,\xi)$, $Y=(y,\eta)$, $Z=(z,\zeta)$ and carrying the symplectic form (\ref{simp}). The composition is given by
\begin{equation}\label{composution}
\left(f\#^0 g\right)(X):=\pi^{-2n}\int_\Xi dY\int_\Xi dZ\,\exp\left[-2i\sigma(X-Y,X-Z)\right]f(Y)g(Z)
\end{equation}
and it is well-known \cite{Fo,GBV} that it satisfies Hypothesis A and B.
One checks easily that
\begin{equation}\label{tracass}
\Theta^0_Z(f):=\e_{-Z}\#^0 f\#^0\e_Z=f\left(\cdot+Z\right),
\end{equation}
so Hypothesis C is verified by a simple change of variables.

The Schr\"odinger representation is given by
\begin{equation}\label{op}
\left[\Op(f)u\right](x):=(2\pi)^{-n}\int_\X\int_{\X'}dy\,d\xi\,\exp\left[i(x-y)\cdot\xi\right]
f\left(\frac{x+y}{2},\xi\right)u(y)
\end{equation}
and it has an excellent behavior and huge applications.
The couple $(\Op,\#^0)$ is the Weyl calculus, a symmetric version of the theory of pseudodifferential
operators \cite{Fo,G1,KN}. Hypothesis X is fulfilled with $\H=L^2(\X)$ and $\G=\S(\X)$.

A direct computation gives
\begin{equation}\label{ase}
\e_{-X}\#^0 f\#^0\e_{X-Y}=e^{\frac{i}{2}\sigma(X,Y)}\e_{-Y}\mathcal T_{\frac{Y}{2}-X}(f)
\end{equation}
and
\begin{equation}\label{version}
\[\M_h(f)\](X,Y)=e^{\frac{i}{2}\sigma(X,Y)}\int_\Xi dZ\,e^{-i\sigma(Y,Z)}h(Z)f(Z+X-Y/2),
\end{equation}
which only differs from the celebrated Short Time Fourier Transform \cite{G1}
\begin{equation}\label{version}
[\mathcal V_{\overline{h}}(f)](X,Y)=\[\mathcal V(f\otimes\overline h)\](X,Y)=
e^{-2\pi iX\cdot Y}\int_\Xi dZ \,e^{-2\pi iZ\cdot Y}\overline{h(Z)}f(Z+X)
\end{equation}
by some conventions. This is relevant in the context of Proposition \ref{etp}, by taking $\mathbb M=\mathcal V$
and $\LL=L^{p,q}(\Xi\times\Xi)$. It can be seen that the assumption of this Proposition are fulfilled, so the modulation spaces
$\[L^{p,q}(\Xi\times\Xi)\]^\M_h=:\mathcal M_h^{p,q}(\Xi)$ one induces by our abstract procedure coincide with the traditional ones.
It is more precise and easier to check by simple changes of variables that $\parallel \M_h(f)\parallel_{L^{p,q}}$ and
$\parallel \mathcal V_h(f)\parallel_{L^{p,q}}$ are proportional if $h\in\S(\Xi)$
is real; recall that $\sigma(X,Y)=X\cdot JY$ for the standard symplectic $2n\times 2n$ matrix $J$.

\subsection{The magnetic Weyl calculus}\label{lizeisc}

It has as a background the problem of quantization of a physical system consisting in a spin-less particle moving in the
euclidean space $\X:=\mathbb R^n$ under the influence of a magnetic field, i.e.
a closed $2$-form $B$ on $\X$ ($dB=0$), given by matrix-component functions
$$
B_{jk}=-B_{kj}:\X\rightarrow\mathbb R,\ \ \ \ j,k=1,\dots,n.
$$
If $a,b,c\in \X$, then we denote by $<a,b,c>$ the triangle in $\X$ of vertices $a,b$ and $c$ and set
$\,\Gamma^B(<a,b,c>):=\int_{<a,b,c>}B\,$ for the flux of $B$ through it (invariant integration of a $2$-form through a $2$-simplex). Then
\begin{equation}\label{composition}
\left(f\#^B g\right)(X):=\pi^{-2n}\int_\Xi dY\int_\Xi dZ\,\exp\left[-2i\sigma(X-Y,X-Z)\right] \times
\end{equation}
$$
\times\exp\left[-i\Gamma^B(<x-y+z,y-z+x,z-x+y>)\right]f(Y)g(Z)
$$
defines a formal associative composition law on functions $f,g:\Xi\rightarrow \mathbb C$.

The formula (\ref{composition}) makes sense and have nice properties under various circumstances. We are going to assume that the
components $B_{jk}$ belong to $C^\infty_{\rm{pol}}(\X)$, the class of smooth functions on $\X$ with polynomial bounds on all the
derivatives. It is shown in \cite{MP1,IMP} (see also \cite{KO1,MPR2}) that, under this assumption, Hypothesis A is true for the
composition law $\#^B$. One checks easily that
\begin{equation}\label{fac}
\e _X\,\#^B\e_Y= \Omega^B(X,Y)\,\#^B\,\e_{X+Y},
\end{equation}
where $\Omega^B$ is the $2$-cocycle defined by the canonical symplectic form and by the magnetic field $B$:
\begin{equation}\label{caf}
\Omega^B(X,Y)(z)\equiv\Omega^B(X,Y;z):=\exp\left[\frac{i}{2}\,\sigma(X,Y)\right]\omega^B(X,Y;z),
\end{equation}
with
\begin{equation}\label{fca}
\omega^B(X,Y;z):=\exp\left[-i\Gamma^B(<z,z+x,z+x+y>)\right].
\end{equation}

In conformity with (\ref{tracas}) we set $\Theta^B_Z(f):=\e_{-Z}\,\#^B f\,\#^B\e_{Z}$
for the family of {\it magnetic translations in phase-space}. They were introduced in \cite{IMP'} and used in
characterizing magnetic pseudodifferential operators by commutators.
We shall need an explicit form of $\Theta^B_Z$, obtained in \cite{IMP'}. For this we define
 the following commutative {\it mixed product} (this is a mixture between point-wise multiplication in the first variable
 and convolution in the second):
\begin{equation}\label{star}
 (f\star g)(x,\xi)\,:=\,\int_{\mathcal{X}'}d\eta\,f(x,\xi-\eta)\,g(x,\eta).
\end{equation}
For any 3 points $x,y,z\in \mathcal{X}$ let us define the parallelogram
\begin{equation}
 \mathcal{P}(x;y,z)\,:=\,\{x+sy+tz\mid s\in[-1/2,1/2],\,t\in[-1,0]\},
\end{equation}
having edges parallel to the vectors $y$ and $z$, respectively. We consider the distribution
\begin{equation}
 \Omega^B[\mathcal{P}(x;y,z)]=\exp\left\{-i\Gamma^B[\mathcal{P}(x;y,z)]\right\}=
 \exp\left\{-i\sum_{j,k=1}^Ny_jz_k\int_{-1/2}^{1/2}ds\int_{-1}^0dtB_{jk}(x+sy+tz)\right\}
\end{equation}
and its Fourier transform with respect to the second variable:
\begin{equation}
 \widetilde{\Omega^B_\mathcal{P}}[z](x,\xi)\,:=(2\pi)^{-n}\,\int_{\mathcal{X}}dy\, e^{-iy\cdot\xi}\Omega^B[\mathcal{P}
 (x;y,z)].
\end{equation}
For $Z=(z,\zeta)\in\Xi$ and $f\in\mathcal{S}(\Xi)$, we have
\begin{equation}\label{aut}
\Theta^B_Z(f)\,=\,\widetilde{\Omega^B_\mathcal{P}}[z]\star\Theta_Z[f].
\end{equation}
Using this, it can be shown that for any $f,g\in \mathcal{S}(\Xi)$ one has
\begin{equation}\label{cere}
\int_\Xi\int_\Xi dYdZ\left[\Theta^B_Z(f)\right](Y)\,g(Y)=\int_\Xi dZ f(Z)\int_\Xi dY g(Y),
\end{equation}
so Hypothesis C is also fulfilled.

We turn now briefly to representations.
Being a closed $2$-form in $\X=\mathbb R^n$, the magnetic field is exact: it can be written as $B=dA$ for some
$1$-form $A$ (called {\it vector potential}). Vector potentials enter by their circulations
$\Gamma^A([x,y]):=\int_{[x,y]}A$ through segments $[x,y]:=\{tx+(1-t)y\mid t\in[0,1]\}$, $x,y\in\X$.
For a vector potential $A$ with $dA=B$, let us define
\begin{equation}\label{op}
\left[\mathfrak{Op}^{A}(f)u\right](x):=(2\pi)^{-n}\int_\X\int_{\X'}dy\,d\xi\,\exp\left[i(x-y)\cdot\xi\right]
\exp\left[-i\Gamma^A([x,y])\right]f\left(\frac{x+y}{2},\xi\right)u(y).
\end{equation}
For $A=0$ one recognizes the standard Weyl quantization. An important property of (\ref{op}) is {\it gauge covariance}:
if $A'=A+d\rho$  defines the same magnetic field as $A$, then $\mathfrak Op^{A'}(f)=e^{i\rho}\Op^{A}(f)e^{-i\rho}$.

It is shown in \cite{MP1} that for any $dA=B$, $\Op^{A}$ defines a representation of the $^*$-algebra
$(\S(\Xi),\#^{B},\overline\cdot)$: Hypothesis X (implying Hypothesis B) holds with $\H=L^2(\mathcal X)$ and $\G=\mathcal S(\mathcal X)$.
Therefore the machine leading to modulation mappings and spaces can be released. It will reproduce the results given in \cite{MP4}
by a direct treatment.

\bigskip

{\bf Acknowledgements:} The authors have been partially supported by {\it N\'ucleo Cientifico ICM P07-027-F "Mathematical Theory of
Quantum and Classical Magnetic Systems"}.

M. M\u{a}ntoiu acknowledges partial support by Chilean Science Foundation {\it Fondecyt}
under the Grant 1085162 and R. Purice acknowledges partial support by CNCSIS under the
Contract PCCE-8/2010 \textit{Sisteme diferentiale in analiza neliniara si aplicatii} and thanks Universidad de Chile for its
kind hospitality during the elaboration of this paper.


\begin{thebibliography}{00}


\bibitem{BB1} I. Beltit\u a and D. Beltit\u a: {\it Magnetic Pseudo-differential Weyl Calculus on Nilpotent Lie Groups.}
Ann. Global Anal. Geom. {\bf 36}, no.3, 293--322, (2009).

\bibitem{BB2} D. Beltit\u a and I. Beltit\u a: {\it Modulation Spaces of Symbols for Representations of Nilpotent Lie
Groups}, J. Fourier Anal. Appl. {\bf 17} (2011), no. 2, 290-319.

\bibitem{BB5} D. Belti\c t\u a and I. Belti\c t\u a: {\it Continuity of Magnetic Weyl Calculs},
J. Funct. Analysis, {\bf 260} (2011), no.7, 1944-1968.

\bibitem{Bou} N. Bourbaki: {\it Espaces vectorieles topologiques, Ch. 1-5}, Springer Berlin Heidelberg New York, (2007).

\bibitem{B1} A. Boulkhemair: {\it Remarks on a Wiener Type Pseudodifferential Algebra and Fourier Integral Operators},
Math. Res. Lett. {\bf 4} (1), 53--67, (1997).

\bibitem{F1} H.~G. Feichtinger: {\it On a New Segal Algebra}, Monatsh. Mat. {\bf 92} (4) 269--289, (1981).

\bibitem{F2} H. G. Feichtinger: {\it Modulation Spaces on Locally Compact Abelian Groups}, In {\it Proceedings of "International
Conference on Wavelets and Applications" 2002}, pages 99--140, Chenai, India. Updated version of a technical report,
University of Viena, (1983).

\bibitem{FG} H. G. Feichtinger and K. Gr\"ochenig: {\it Banach Spaces Related to Integrable Group Representations
and their Atomic Decompositions}, I, J. Funct. Anal. {\bf 86}, 307--340, (1989)

\bibitem{FGL} G. Fendler, K. Gr\"ochenig and M. Leinert: {\it Convolution-Dominated Operators on Discrete Groups},
Integr. Equ. Oper. Th. {\bf 61} (2008), 493--509.

\bibitem{Fo} G. B. Folland: {\it Harmonic Analysis in Phase Space}, Princeton Univ. Press, Princeton, NJ, (1989).

\bibitem{G1} K. Gr\"ochenig: {\it Foundations of Time-Frequency Analysis}, Birkh\"auser Boston Inc., Boston, MA,
(2001).

\bibitem{G2} K. Gr\"ochenig: {\it Time-Frequency Analysis of Sj\"ostrand Class}, Revista Mat. Iberoam. {\bf 22} (2),
703--724, (2006).

\bibitem{G3} K. Gr\"ochenig: {\it Composition and Spectral Invariance of Pseudodifferential Operators on Modulation
Spaces}, J. Anal. Math. {\bf 98}, 65--82, (2006).

\bibitem{G4} K. Gr\"ochenig: {\it A Pedestrian Approach to Pseudodifferential Operators} in C. Heil editor,
{\it Harmonic Analysis and Applications}, Birkh\"auser, Boston, 2006. In Honour of John J. Benedetto.

\bibitem{GS} K. Gr\"ochenig and T. Sthromer: {\it Pseudodifferential Operators on Locally Compact Abelian Groups
and Sj\"ostrand's Symbol Class}, J. Reine Angew. Math. {\bf 613}, 121--146, (2007).

\bibitem{GBV} J. M Gracia Bondia and J. C. Varilly: {\it Algebras of Distributions Suited to Phase-Space Quantum Mechanics. I},
J. Math. Phys, {\bf 29}, 869--878, (1988).

\bibitem{Ho} L. H\"ormander: {\it The Analysis of Linear Partial Differential Operators}, III, Spriger-Verlag, Berlin Heidelberg New York
Tokio, (1985).

\bibitem{HRT} C. Heil, J. Ramanathan, and P. Topiwala: {\it Singular Values of Compact Pseudodifferential Operators},
J. Funct. Anal., {\bf 150} (2), 426--452, (1997).

\bibitem{IMP} V. Iftimie, M. M\u antoiu and R. Purice: {\it Magnetic Pseudodifferential Operators},
Publ. RIMS. {\bf 43}, 585--623, (2007).

\bibitem{IMP'} V. Iftimie, M. M\u antoiu and R. Purice: {\it A Beals-Type Criterion for Magnetic Pseudodifferential
Operators}, Comm. in PDE {\bf 35} (6) (2010), 1058--1094.

\bibitem{KO1} M.V. Karasev and T.A. Osborn: {\it Symplectic Areas,
Quantization and Dynamics in Electromagnetic Fields}, J. Math. Phys. {\bf 43}, 756--788, (2002).

\bibitem{KN} J. Kohn, L. Nirenberg: \textit{An Algebra of Pseudodifferential Operators},  Commun. Pure Appl. Math., \textbf{18},
pp. 269--305 (1965)

\bibitem{MP1} M. M\u antoiu and R. Purice: {\it The Magnetic
Weyl Calculus}, J. Math. Phys. {\bf 45}, No 4, 1394--1417 (2004).

\bibitem{MP4} M. M\u antoiu and R. Purice: {\it The Modulation Mapping for Magnetic Symbols and Operators},
Proc. of the AMS, DOI: 10.1090/S0002-9939-10-10345-1 (2010).

\bibitem{MPR2} M. M\u antoiu, R. Purice and S. Richard: {\it Spectral and Propagation Results
for Magnetic Schr\"odinger Operators; a $C^*$-Algebraic Approach}, J. Funct. Anal. {\bf 250}, 42--67, (2007).

\bibitem{Pe} N. V. Pedersen: {\it Matrix Coefficients and a Weyl Correspondence for Nilpotent Lie Groups}, Invent. Math.
{\bf 118}, 1--36, (1994).

\bibitem{S1} J. Sj\"ostrand: {\it An Algebra of Pseudodifferential Operators}, Math. Res. Lett., {\bf 1} (2), 185--192,
(1994).

\bibitem{S2} J. Sj\"ostrand: {\it Wiener Type Algebras of Pseudodifferential Operators}, In {\it S\'eminaire sur les \'Equations
aux D\'eriv\'ees Partielles, 1994-1995}, Exp. No. IV, 21, \'Ecole Polytech., Palaiseau, (1995).

\bibitem{T1} J. Toft: {\it Subalgebras to a Wiener Type Algebra of Pseudo-differential Operators}, Ann. Inst. Fourier
(Grenoble), {\bf 51} (5), 1347---1383, (2001).

\bibitem{T3} J. Toft: {\it Continuity Properties for Non-commutative Convolution Algebras with Applications in Pseudo-differential Calculus},
Bull. Sci. Math. {\bf 126} (2), 115--142, (2002).

\bibitem{T4} J. Toft: {\it Positive Properties for Non-commutative Convolution Algebras with Applications in Pseudo-differential Calculus},
Bull. Sci. Math. {\bf 127} (2), 101--132, (2003).

\bibitem{T2} J. Toft: {\it Continuity Properties for Modulation Spaces with Applications to
Pseudo-differential Calculus. I}, J. Funct. Anal., {\bf 207} (2), 399--426, (2004).

\bibitem{Tr} F. Tr\`eves: {\it Topological Vector Spaces, Distributions and Kernels}, Academic Press, 1967.

\bibitem{W} D. Williams: {\it Crossed Products of $C^*$-Algebras}, Mathematical Surveys and Monographs, {\bf 134},
American Mathematical Society, 2007.

\end{thebibliography}
\end{document}